\documentclass[reqno,12pt]{amsart}

\newtheorem{theorem}{Theorem}[section]
\newtheorem{lemma}[theorem]{Lemma}

\theoremstyle{definition}

\newtheorem{example}[theorem]{Example}

\theoremstyle{remark}

\numberwithin{equation}{section}

\newcommand{\ta}{\theta}
\newcommand{\ty}{\infty}

\begin{document}

\title[An elliptic Taylor expansion theorem]{A Taylor expansion
theorem for an elliptic extension of the Askey--Wilson operator}

\author{Michael J.\ Schlosser}
\address{Fakult\"at f\"ur Mathematik, Universit\"at Wien,
Nordbergstra{\ss}e 15, A-1090 Vienna, Austria}
\email{michael.schlosser@univie.ac.at}
\urladdr{http://www.mat.univie.ac.at/{\textasciitilde}schlosse}
\thanks{The author was partly supported by FWF Austrian Science Fund
grants \hbox{P17563-N13}, and S9607 (the latter is part
of the Austrian National Research Network
``Analytic Combinatorics and Probabilistic Number Theory'').}

\subjclass{Primary 33D15; Secondary 33D45, 33E05, 33E20, 41A30, 41A58}
\date{February 1, 2008}

\keywords{Askey--Wilson operator,
elliptic hypergeometric series,
Frenkel and Turaev's ${}_{10}V_9$ summation,
Frenkel and Turaev's ${}_{12}V_{11}$ transformation,
nonterminating ${}_8\phi_7$ summation,
elliptic biorthogonal functions,
Taylor expansions}

\begin{abstract}
We establish Taylor series expansions in rational (and elliptic)
function bases using E.~Rains' elliptic extension of the Askey--Wilson
divided difference operator. The expansion theorem we consider extends
M.~E.~H.~Ismail's expansion for the Askey--Wilson monomial basis.
Three immediate applications (essentially already due to Rains)
include simple proofs of Frenkel and Turaev's elliptic extensions of
Jackson's ${}_8\phi_7$ summation and of Bailey's ${}_{10}\phi_9$
transformation, and the computation of the connection
coefficients of Spiridonov's elliptic extension of Rahman's
biorthogonal rational functions. We adumbrate other examples
including the nonterminating extension of Jackson's ${}_8\phi_7$
summation and a quadratic expansion.
\end{abstract}

\maketitle

\section{Introduction}

Taylor series expansion is a powerful and well-known tool in analysis
for studying the local behaviour of a suitable function
(being approximated by its partial Taylor sums).
The explicit expansion of a function in terms of another
given basis of the function space is on one hand an important concept
in harmonic analysis, and on the other hand,
from a more algebraic point of view, it is simply
a fundamental technique for obtaining identities, which,
for instance, is one of the main ideas of umbral calculus~\cite{RKO}.

In \cite{Is1}, Ismail gave a Taylor expansion theorem
involving the Askey--Wilson divided difference operator.
He utilized it to give simple proofs of the $q$-Pfaff--Saalsch\"utz summation,
and of the more general Sears transformation (relating two Askey--Wilson
polynomials). More $q$-Taylor expansions related to the Askey--Wilson
operator were given in \cite{IS} and later in \cite{MP}
(to cite just a few relevant papers).
As a matter of fact, none of the expansions obtained in the
aforementioned papers involved {\em well-poised} series.
Such expansions and related difference operators were however
considered (in more generality, namely in the setting of
{\em multivariate elliptic hypergeometric series})
by Rains~\cite{Ra1}, \cite{Ra2}, and were also investigated by
Rosengren~\cite{R1}, \cite{R2}.

The purpose of the present paper is two-fold.
Although the elliptic Taylor expansion in Theorem~\ref{thnew}
has not been stated explicitly before (to the author's knowledge),
it is implicit from the (more general) work of Rains~\cite{Ra1}
who ad-hoc also gave corresponding applications.
Regarding the much higher level of generality and complexity
of results in \cite{Ra1}, it appears (to the present author)
that these results, even in their simplest noteworthiest cases
can easily be missed by non-specialists (who are maybe not
so much interested in the multivariate theory which requires a
more elaborate setup).
Therefore one of our intentions is to make these results easy accessible.
The other aim is to announce some new applications concerning
infinite Taylor expansions in the non-elliptic case involving
well-poised basic series; see the final section.

\section{Preliminaries}

For the following material, we refer to Gasper and Rahman's text
\cite{GR}. Throughout this paper, we assume $q$ to be a fixed complex
number satisfying $|q|<1$.

\subsection{Basic hypergeometric series}

For any complex number $a$ and integer $n$ the $q$-shifted
factorial is defined by
\begin{equation}\label{defpoch}
(a;q)_n = \frac{(a;q)_\ty}{(aq^n;q)_\ty},\qquad\text{where}\qquad
(a;q)_\ty = \prod_{j\ge 0}(1-aq^j).
\end{equation}
For products of $q$-shifted factorials we use the short notation
\begin{equation*}
(a_1, a_2, \ldots, a_m;q)_n = \prod^m_{k=1} (a_k;q)_n,
\end{equation*}
where $n$ is an integer or infinity.
A list of useful identities for manipulating the
$q$-shifted factorials is given in \cite[Appendix~I]{GR}.

We use
\begin{equation}\label{defhyp}
{}_{s+1}\phi_s\!\left[\begin{matrix}a_1,a_2,\dots,a_{s+1}\\
b_1,b_2,\dots,b_s\end{matrix}\,;q,z\right]:=
\sum _{k=0} ^{\infty}\frac {(a_1,a_2,\dots,a_{s+1};q)_k}
{(q,b_1,\dots,b_s;q)_k}z^k
\end{equation}
to denote the {\em basic hypergeometric ${}_{s+1}\phi_s$ series}.
In \eqref{defhyp}, $a_1,\dots,a_{s+1}$ are
called the {\em upper parameters}, $b_1,\dots,b_s$ the
{\em lower parameters}, $z$ is the {\em argument}, and 
$q$ the {\em base} of the series.
The ${}_{s+1}\phi_s$ series terminates if one of the upper parameters,
say $a_{s+1}$, is of the form $q^{-n}$ for a nonnegative integer $n$.
If the ${}_{s+1}\phi_s$ series does not terminate, it converges when $|z|<1$.

The classical theory of basic hypergeometric series contains
numerous summation and transformation formulae
involving ${}_{s+1}\phi_s$ series.
Many of these summation theorems require
that the parameters satisfy the condition of being
either balanced and/or very-well-poised.
An ${}_{s+1}\phi_s$ basic hypergeometric series is called
{\em balanced} if $b_1\cdots b_s=a_1\cdots a_{s+1}q$ and $z=q$.
An ${}_{s+1}\phi_s$ series is {\em well-poised} if
$a_1q=a_2b_1=\cdots=a_{s+1}b_s$. An ${}_{s+1}\phi_s$ basic
hypergeometric series is called {\em very-well-poised}
if it is well-poised and if $a_2=-a_3=q\sqrt{a_1}$.
Note that the factor
\begin{equation*}
\frac {1-a_1q^{2k}}{1-a_1}
\end{equation*}
appears in a very-well-poised series.
The parameter $a_1$ is usually referred to as the
{\em special parameter} of such a series.

One of the most important theorems in the theory of basic
hypergeometric series is 
Jackson's~\cite{J} terminating very-well-poised balanced
${}_8\phi_7$ summation (cf.\ \cite[Eq.~(2.6.2)]{GR}):
\begin{multline}\label{87gl}
{}_8\phi_7\!\left[\begin{matrix}a,\,q\sqrt{a},-q\sqrt{a},b,c,d,
a^2q^{1+n}/bcd,q^{-n}\\
\sqrt{a},-\sqrt{a},aq/b,aq/c,aq/d,bcdq^{-n}/a,aq^{1+n}\end{matrix}\,;q,
q\right]\\
=\frac {(aq,aq/bc,aq/bd,aq/cd;q)_n}
{(aq/b,aq/c,aq/d,aq/bcd;q)_n}.
\end{multline}
This identity stands on the top of the classical hierarchy of summations
for basic hypergeometric series. Special cases include the terminating
and nonterminating very-well-poised ${}_6\phi_5$ summations,
the $q$-Pfaff--Saalsch\"utz summation, the $q$-Gau{\ss} summation,
the $q$-Chu--Vandermonde summation and the termininating and
nonterminating $q$-binomial theorem, see \cite{GR}.

\subsection{Elliptic hypergeometric series}

Here, we refer to Chapter~11 of
Gasper and Rahman's text \cite{GR}. Define a modified
Jacobi theta function with argument $x$ and nome $p$ by
\begin{equation}\label{mjtf}
\ta (x; p):= (x; p)_\ty (p/x; p)_\ty\,,\quad\quad
\ta (x_1, \ldots, x_m;p) = \prod^m_{k=1} \ta (x_k;p),
\end{equation}
where $ x, x_1, \ldots, x_m \ne 0,\ |p| < 1$.
We note the following useful properties of theta functions:
\begin{equation}\label{tif}
\ta(x;p)=-x\,\ta(1/x;p),
\end{equation}
\begin{equation}\label{p1id}
\ta(px;p)=-\frac 1x\,\ta(x;p),
\end{equation}
and Riemann's {\em addition formula}
\begin{equation}\label{addf}
\ta(xy,x/y,uv,u/v;p)-\ta(xv,x/v,uy,u/y;p)
=\frac uy\,\ta(yv,y/v,xu,x/u;p)
\end{equation}
(cf.\ \cite[p.~451, Example 5]{WW}).

Further, define a {\em theta shifted factorial} analogue of the
$q$-shifted factorial by
\begin{equation}\label{defepoch}
(a;q,p)_n = \begin{cases}
\prod^{n-1}_{k=0} \ta (aq^k;p),& n = 1, 2, \ldots\,,\cr
1,& n = 0,\cr
1/\prod^{-n-1}_{k=0} \ta (aq^{n+k};p),& n = -1, -2, \ldots,
\end{cases}
\end{equation}
and let
\begin{equation*}
(a_1, a_2, \ldots, a_m;q, p)_n = \prod^m_{k=1} (a_k;q,p)_n,
\end{equation*}
where $a, a_1,\ldots,a_m \neq 0$.
Notice that $\ta (x;0) = 1-x$ and, hence, $(a;q, 0)_n = (a;q)_n$
is a $q$-{\it shifted factorial} in base $q$.
The parameters $q$ and $p$ in $(a;q,p)_n$ are called the
{\it base} and {\it nome}, respectively, and
$(a;q,p)_n$ is called the $q,p$-{\it shifted factorial}.
Observe that
\begin{equation}\label{pid}
(pa;q,p)_n=(-1)^na^{-n}q^{-\binom n2}\,(a;q,p)_n,
\end{equation}
which follows from \eqref{p1id}. 
A list of other useful identities for manipulating the
$q,p$-shifted factorials is given in \cite[Sec.~11.2]{GR}.

We call a series $\sum c_n$ an {\it elliptic hypergeometric series} if
$g(n) = c_{n+1}/c_n$ is an elliptic function of $n$ with $n$
considered as a complex variable; i.e., the function $g(x)$
is a doubly periodic meromorphic function of the complex variable $x$.
Without loss of generality, by the theory of theta functions,
we may assume that
\begin{equation*}
g(x)=\frac{\ta(a_1q^x,a_2q^x,\dots,a_{s+1}q^x;p)}
{\ta(q^{1+x},b_1q^x,\dots,b_sq^x;p)}\,z,
\end{equation*}
where the {\em elliptic balancing condition}, namely
$$a_1a_2\cdots a_{s+1}=qb_1b_2\cdots b_s,$$
holds.
If we write $q=e^{2\pi i\sigma}$, $p=e^{2\pi i\tau}$,
with complex $\sigma$, $\tau$, then $g(x)$ is indeed periodic in $x$
with periods $\sigma^{-1}$ and $\tau\sigma^{-1}$.

The general form of an elliptic hypergeometric series is thus
\begin{equation*}
{}_{s+1}E_s\!\left[\begin{matrix}a_1,\dots,a_{s+1}\\
b_1,\dots,b_s\end{matrix};q,p;z\right]
:=\sum_{k=0}^{\infty}\frac
{(a_1,a_2,\dots,a_{s+1};q,p)_k}{(q,b_1\dots,b_s;q,p)_k}z^k,
\end{equation*}
provided $a_1a_2\cdots a_{s+1}=qb_1b_2\cdots b_s$.
Here $a_1,\dots,a_{s+1}$ are the upper parameters,
$b_1,\dots,b_s$ the lower parameters, $q$ is the base,
$p$ the nome, and $z$ is the argument of the series.
For convergence reasons, one usually requires $a_{s+1}=q^{-n}$
($n$ being a nonnegative integer),
so that the sum is in fact finite.

{\em Very-well-poised elliptic hypergeometric series} are defined as
\begin{multline}\label{vwpehs}
{}_{s+1}V_s(a_1;a_6,\dots,a_{s+1};q,p;z)\\:=
{}_{s+1}E_s\!\left[\begin{matrix}{\scriptstyle a_1},\,
{\scriptstyle qa_1^{\frac 12}},{\scriptstyle -qa_1^{\frac 12}},
{\scriptstyle qa_1^{\frac 12}/p^{\frac 12}},
{\scriptstyle -qa_1^{\frac 12}p^{\frac 12}},{\scriptstyle a_6},
\dots,{\scriptstyle a_{s+1}}\\
{\scriptstyle a_1^{\frac 12}},{\scriptstyle -a_1^{\frac 12}},
{\scriptstyle a_1^{\frac 12}p^{\frac 12}},
{\scriptstyle -a_1^{\frac 12}/p^{\frac 12}},
{\scriptstyle a_1q/a_6},\dots,
{\scriptstyle a_1q/a_{s+1}}\end{matrix};q,p;-z\right]\\
=\sum_{k=0}^{\infty}\frac{\ta(a_1q^{2k};p)}{\ta(a_1;p)}
\frac{(a_1,a_6,\dots,a_{s+1};q,p)_k}
{(q,a_1q/a_6,\dots,a_1q/a_{s+1};q,p)_k}(qz)^k,
\end{multline}
where
\begin{equation*}
q^2a_6^2a_7^2\cdots a_{s+1}^2=(a_1q)^{s-5}.
\end{equation*}
It is convenient to abbreviate
\begin{equation*}
{}_{s+1}V_s(a_1;a_6,\dots,a_{s+1};q,p)
:={}_{s+1}V_s(a_1;a_6,\dots,a_{s+1};q,p;1).
\end{equation*}
Note that in \eqref{vwpehs} we have used
\begin{equation*}
\frac{\ta(aq^{2k};p)}{\ta(a;p)}=
\frac{(qa^{\frac 12},-qa^{\frac 12},qa^{\frac 12}/p^{\frac 12},
-qa^{\frac 12}p^{\frac 12};q,p)_k}
{(a^{\frac 12},-a^{\frac 12},
a^{\frac 12}p^{\frac 12},-a^{\frac 12}/p^{\frac 12};q,p)_k}
(-q)^{-k},
\end{equation*}
which shows that in the elliptic case the number of
pairs of numerator and denominator paramters
involved in the construction of the {\em very-well-poised term} is {\em four}
(whereas in the basic case this number is {\em two},
in the ordinary case only {\em one}).

The above definitions for ${}_{s+1}E_{s}$ and ${}_{s+1}V_{s}$ series
are due to Spiridonov~\cite{Sp}, see \cite[Ch.~11]{GR}.

In their study of elliptic $6j$ symbols (which are elliptic solutions
of the Yang--Baxter equation found by Baxter~\cite{B} and Date et
al.~\cite{DJKMO}), Frenkel and Turaev~\cite{FT}
discovered the following ${}_{12}V_{11}$ transformation:
\begin{multline}\label{12V11}
{}_{12}V_{11}(a;b,c,d,e,f,\lambda a q^{n+1}/ef,q^{-n};q,p)
=\frac {(aq,aq/ef,\lambda q/e,\lambda q/f;q,p)_n}
{(aq/e,aq/f,\lambda q/ef,\lambda q;q,p)_n}\\\times
{}_{12}V_{11}(\lambda;\lambda b/a,\lambda c/a,\lambda d/a,e,f,
\lambda a q^{n+1}/ef,q^{-n};q,p),
\end{multline}
where $\lambda=a^2q/bcd$. This is an extension of Bailey's
very-well-poised ${}_{10}\phi_9$ transformation~\cite[Eq.~(2.9.1)]{GR},
to which it reduces when $p=0$.

The ${}_{12}V_{11}$ transformation in \eqref{12V11} appeared as 
a consequence of the tetrahedral symmetry of the elliptic $6j$ symbols.
Frenkel and Turaev's transformation contains as a special case
the following summation formula,
\begin{equation}\label{10V9}
{}_{10}V_9(a;b,c,d,e,q^{-n};q,p)
=\frac {(aq,aq/bc,aq/bd,aq/cd;q,p)_n}
{(aq/b,aq/c,aq/d,aq/bcd;q,p)_n},
\end{equation}
where $a^2q^{n+1}=bcde$.
The $_{10}V_9$ summation is an elliptic analogue of Jackson's
$_8\phi_7$ summation formula \eqref{87gl}.
A striking feature of elliptic hypergeometric series is that
already the simplest identities involve many parameters.
The fundamental identity at the ``bottom'' of the hierarchy of
identities for elliptic hypergeometric series is the $_{10}V_9$ summation.
When keeping the nome $p$ arbitrary (while $|p|<1$) there is no way
to specialize (for the sake of obtaining lower order identities)
any of the free parameters of an elliptic hypergeometric series
in form of a limit tending to zero or infinity, due to the issue
of convergence. For the same reason, elliptic hypergeometric series
are only well-defined as complex functions if they are terminating
(i.e., the sums are finite).
See Gasper and Rahman's text \cite[Ch.~11]{GR} for more details.

\section{The Askey--Wilson operator}

We will be considering meromorphic functions $f(z)$ symmetric in $z$
and $1/z$. Writing $z=e^{i\theta}$ (note that $\theta$ need not be real),
we may consider $f$ to be a function in $x=\cos\theta=(z+1/z)/2$ and write
$f[x]:=f(z)$.

Let $\mathcal D_q$ denote the {\em Askey--Wilson operator} acting on
$x=\cos\theta$. It is defined as follows:
\begin{equation}\label{awop}
\mathcal D_q f[x]=\frac{f(q^{\frac 12}z)-f(q^{-\frac 12}z)}
{\iota(q^{\frac 12}z)-\iota(q^{-\frac 12}z)},
\end{equation}
where $\iota[x]=x$ (i.e., $\iota(z)=(z+1/z)/2$).
If follows from \eqref{awop} that
\begin{equation}\label{awop1}
\mathcal D_q f[x]=\frac{f(q^{\frac 12}z)-f(q^{-\frac 12}z)}
{i(q^{\frac 12}-q^{-\frac 12})\sin\theta}.
\end{equation}
The operator $\mathcal D_q$ was introduced in \cite{AW} and is a
$q$-analogue of the differentiation operator. In particular, since 
\begin{equation}\label{awop2}
\mathcal D_q T_n[x]=
\frac{q^{\frac n2}-q^{-\frac n2}}
{q^{\frac 12}-q^{-\frac 12}}U_{n-1}[x],
\end{equation}
where $T_n[\cos\theta]=\cos n\theta$ and
$U_n[\cos\theta]=\sin(n+1)\theta/\sin\theta$ are the Chebyshev polynomials
of the first and second kind, one easily sees that $\mathcal D_q$ maps
polynomials to polynomials, lowering the degree by one.

In the calculus of the Askey--Wilson operator the so-called
``Askey--Wilson monomials'' $\phi_n(x;a)=(az,a/z;q)_n$
form a natural basis for polynomials or power series in $x$.
One readily computes
\begin{equation}\label{awop3}
\mathcal D_q (az,a/z;q)_n=-\frac{2a (1-q^n)}{(1-q)}
(aq^{\frac 12}z,aq^{\frac 12}/z;q)_{n-1}.
\end{equation}
We recall the following Taylor theorem for polynomials
$f[x]$, proved by Ismail~\cite{Is1}:
\begin{theorem}\label{thism}
If $f[x]$ is a polynomial in $x$ of degree $n$, then
\begin{equation*}
f[x]=\sum_{k=0}^nf_k\phi_k(x;a),
\end{equation*}
where
\begin{equation*}
f_k=\frac{(q-1)^k}{(2a)^k(q;q)_k}q^{-k(k-1)/4}
\big(\mathcal D_q^k f\big)[x_k], \qquad
x_k:=\frac 12(aq^{\frac k2}+q^{-\frac k2}/a).
\end{equation*}
\end{theorem}
As was shown in \cite{Is1},
the application of Theorem~\ref{thism} to $f(z)=(bz,b/z;q)_n$
immediately gives the $q$-Pfaff--Saalsch\"utz summation
(cf.\ \cite[Eq.~(1.7.2)]{GR}), in the form
\begin{equation*}
\frac{(bz,b/z;q)_n}{(ba,b/a;q)_n}=
{}_3\phi_2\!\left[\begin{matrix}az,a/z,q^{-n}\\
ab,q^{1-n}a/b\end{matrix};q,q\right],
\end{equation*}
while its application to
the Askey--Wilson polynomials,
\begin{equation*}
\omega_n(x;a,b,c,d;q):=
{}_4\phi_3\!\left[\begin{matrix}az,a/z,abcdq^{n-1},q^{-n}\\
ab,ac,ad\end{matrix};q,q\right],
\end{equation*}
gives a connection coefficient identity which,
by specialization, can be reduced to the Sears transformation
(cf.\ \cite[Eq.~(3.2.1)]{GR}), in the form
\begin{equation*}
\omega_n(x;a,b,c,d;q)=
\frac{a^n(bc,bd;q)_n}{b^n(ac,ad;q)_n}\,
\omega_n(x;b,a,c,d;q).
\end{equation*}

Ismail and Stanton~\cite{IS} extended the above polynomial Taylor theorem
to hold for entire functions of exponential growth,
resulting in infinite Taylor expansions.
Marco and Parcet~\cite{MP} extended this yet further to hold for
arbitrary $q$-differentiable functions, resulting in
infinite Taylor expansions with explicit remainder term.
Among other results they were able to
recover the nonterminating $q$-Pfaff--Saalsch\"utz summation
(cf.\ \cite[Appendix~(II.24)]{GR}).

\section{A well-poised and elliptic Askey--Wilson operator}
Since
\begin{multline*}
\mathcal D_q \frac{(az,a/z;q)_n}{(cz,c/z;q)_n}\\
=\frac 2{(q^{\frac 12}-q^{-\frac 12})(z-1/z)}
\left[\frac{(aq^{\frac 12}z,aq^{-\frac 12}/z;q)_n}
{(cq^{\frac 12}z,cq^{-\frac 12}/z;q)_n}-
\frac{(aq^{-\frac 12}z,aq^{\frac 12}/z;q)_n}
{(cq^{-\frac 12}z,cq^{\frac 12}/z;q)_n}\right]\\=
\frac 2{(q^{\frac 12}-q^{-\frac 12})(z-1/z)}
\frac{(aq^{\frac 12}z,aq^{\frac 12}/z;q)_{n-1}}
{(cq^{\frac 12}z,cq^{\frac 12}/z;q)_{n-1}}\\\times
\left[\frac{(1-azq^{n-\frac 12})(1-aq^{-\frac 12}/z)}
{(1-czq^{n-\frac 12})(1-cq^{-\frac 12}/z)}-
\frac{(1-azq^{-\frac 12})(1-aq^{n-\frac 12}/z)}
{(1-czq^{-\frac 12})(1-cq^{n-\frac 12}/z)}
\right]\\=
\frac{(-1)2a(1-c/a)(1-acq^{n-1})(1-q^n)}
{(1-czq^{-\frac 12})(1-czq^{\frac 12})
(1-cq^{-\frac 12}/z)(1-cq^{\frac 12}/z)(1-q)}
\frac{(aq^{\frac 12}z,aq^{\frac 12}/z;q)_{n-1}}
{(cq^{\frac 32}z,cq^{\frac 32}/z;q)_{n-1}},
\end{multline*}
it makes sense to define a $c$-generalized  Askey--Wilson operator
acting on $x$ (or $z$) by
\begin{equation*}
\mathcal D_{c,q}=(1-czq^{-\frac 12})(1-czq^{\frac 12})
(1-cq^{-\frac 12}/z)(1-cq^{\frac 12}/z)\,\mathcal D_q,
\end{equation*}
which acts ``degree-lowering'' on the ``rational monomials''
\begin{equation*}
\frac{(az,a/z;q)_n}{(cz,c/z;q)_n}
\end{equation*}
in the form
\begin{multline*}
\mathcal D_{c,q} \frac{(az,a/z;q)_n}{(cz,c/z;q)_n}=
\frac{(-1)2a(1-c/a)(1-acq^{n-1})(1-q^n)}{(1-q)}
\frac{(aq^{\frac 12}z,aq^{\frac 12}/z;q)_{n-1}}
{(cq^{\frac 32}z,cq^{\frac 32}/z;q)_{n-1}}.
\end{multline*}
Clearly,
\begin{equation*}
\mathcal D_{0,q}=\mathcal D_{q}.
\end{equation*}

More generally, for parameters $c,q,p$ with $|q|,|p|<1$, we define an
elliptic extension of the Askey--Wilson operator, acting on
functions symmetric in $z^{\pm1}$, by
\begin{equation}
\mathcal D_{c,q,p} f(z)
=2q^{\frac 12}z\,\frac{\theta(czq^{-\frac 12},czq^{\frac 12},
cq^{-\frac 12}/z,cq^{\frac 12}/z;p)}{\theta(q,z^2;p)}
\left(f(q^{\frac 12}z)-f(q^{-\frac 12}z)\right).
\end{equation}
Note that
\begin{equation*}
\mathcal D_{c,q,0}=\mathcal D_{c,q}.
\end{equation*}

In particular, using \eqref{addf}, we have
\begin{equation}\label{degl}
\mathcal D_{c,q,p} \frac{(az,a/z;q,p)_n}{(cz,c/z;q,p)_n}=
\frac{(-1)2a\,\theta(c/a,acq^{n-1},q^n;p)}{\theta(q;p)}
\frac{(aq^{\frac 12}z,aq^{\frac 12}/z;q,p)_{n-1}}
{(cq^{\frac 32}z,cq^{\frac 32}/z;q,p)_{n-1}}.
\end{equation}

The operator $\mathcal D_{c,q,p}$ is a special case of some multivariable
difference operator introduced by Rains in \cite{Ra1}.
Already in the single variable case Rains' operator
involves two more parameters than $\mathcal D_{c,q,p}$.
(Rains' difference operators generate a representation of
the Sklyanin algebra, as observed in \cite{Ra1} and made
explicit in \cite{R1} and \cite[Sec.~6]{R2}.)
Rains' operator can be specialized to act degree-lowering
(as the above $\mathcal D_{c,q,p}$ does), degree-preserving
or degree-raising on abelian functions.
Rains used his multivariable difference operators in \cite{Ra1}
to construct $BC_n$-symmetric biorthogonal abelian functions that
generalize Koornwinder's orthogonal polynomials.
He further used his operator in \cite{Ra2} to derive $BC_n$-symmetric
extensions of Frenkel and Turaev's ${}_{10}V_9$ summation
and ${}_{12}V_{11}$ transformation.

For the current presentation, as we are mainly concerned
with Taylor expansions, we find it indeed sufficient to consider
the above operator $\mathcal D_{c,q,p}$ (which exhibits a very nice 
degree-lowering action in \eqref{degl}) rather than the
more general operator considered by Rains (in one dimension).

We describe the spaces of functions we will be dealing with. 
For a complex number $c$ we define
\begin{equation*}
W_c^m:=\operatorname{span}\left\{\frac{g_n(z)}
{(cz,c/z;q,p)_n},\;0\le n\le m\right\},
\end{equation*}
where $g_n(z)$ runs over all functions being holomorphic 
for $z\neq 0$ with $g_n(z)=g_n(1/z)$ and
\begin{equation*}
g_n(pz)=\frac{1}{p^nz^{2n}}g_n(z).
\end{equation*}
In classical terminology, $g_n(z)$ is an even theta function
of order $2n$ and zero characteristics. Rains~\cite{Ra2}
refers to such functions as $BC_1$ theta functions of degree $n$,
whereas in \cite{RS} we referred to them as $D_n$ theta functions.
It is well-known that the space $V^n$ of even theta functions
of order $2n$ and zero characteristics has dimension $n+1$
(see e.g.\ Weber~\cite[p.~49]{W}).

Note that $W_c^m$ consists of certain abelian functions.
(For $p\to 0$ these degenerate to certain rational functions
we may call ``well-poised''.)

\begin{lemma}\label{Lem}
For any arbitrary but fixed complex number $a$
(satisfying $a\neq c q^jp^k$, for $j=0,\dots,m-1$,
and $k\in\mathbb Z$, and $a\neq q^jp^k/c$,
for $j=2-2m,\dots,1-m$, and $k\in\mathbb Z$), the set
\begin{equation*}
\left\{\frac{(az,a/z;q,p)_n}
{(cz,c/z;q,p)_n},0\le n\le m\right\}
\end{equation*}
forms a basis for $W_c^m$.
\end{lemma}

\begin{proof}
This is equivalent to the fact that the set
\begin{equation}\label{basisv}
\left\{(az,a/z;q,p)_n(cq^nz,cq^n/z;q,p)_{m-n},\;0\le n\le m\right\}
\end{equation}
forms a basis for $V^m$, the space of even theta functions
of order $2m$ and zero characteristics, a fact easily proved by
induction on $m$. For $m=0$ the statement is trivial.
Now assume that it holds for a fixed $m\ge 0$. Since the
$m+1$ products in \eqref{basisv}
are linearly independent, it follows (by multiplication with the
common factor $\ta(cq^mz,cq^m/z;p)$) that
the $m+1$ products
\begin{equation*}
(az,a/z;q,p)_n(cq^nz,cq^n/z;q,p)_{m+1-n},\quad 0\le n\le m,
\end{equation*}
are also linearly independent. It thus remains to be shown that
$(az,a/z;q,p)_{m+1}$ is not a linear combination of
$\{(az,a/z;q,p)_n(cq^nz,cq^n/z;q,p)_{m+1-n},\; 0\le n\le m\}$.
Suppose
\begin{equation*}
(az,a/z;q,p)_{m+1}=\sum_{n=0}^m\alpha_n
(az,a/z;q,p)_n(cq^nz,cq^n/z;q,p)_{m+1-n}.
\end{equation*}
Letting $z=cq^m$ gives $(acq^m,aq^{-m}/c;q,p)_{m+1}=0$,
which is a contradiction.
\end{proof}

Note that, in view of \eqref{degl},
the elliptic Askey--Wilson operator maps functions in $W_c^m$ to
functions in $W_{cq^{\frac 32}}^{m-1}$.

We now define
\begin{equation}\label{defd}
\mathcal D_{c,q,p}^{(k)}=\mathcal D_{cq^{\frac 32},q,p}^{(k-1)}\;
\mathcal D_{c,q,p},
\end{equation}
with $\mathcal D_{c,q,p}^{(0)}=\varepsilon$, the identity operator.
We have the following elliptic expansion theorem which extends
Theorem~\ref{thism}:
\begin{theorem}\label{thnew}
Let $f$ be in $W_c^n$, then
\begin{equation}\label{exp}
f(z)=\sum_{k=0}^nf_k\,\frac{(az,a/z;q,p)_k}
{(cz,c/z;q,p)_k},
\end{equation}
where
\begin{equation*}
f_k=\frac{(-1)^kq^{-k(k-1)/4}\,\ta(q;p)^k}{(2a)^k(q,c/a,acq^{k-1};q,p)_k}
\big[\mathcal D_{c,q,p}^{(k)} f\big]_{z=aq^{\frac k2}}.
\end{equation*}
\end{theorem}
\begin{proof}
First of all, due to Lemma~\ref{Lem} it is clear that the expansion
\eqref{exp} exists, so we just need to compute the coefficients $f_k$.
Formula \eqref{degl} yields (together with \eqref{defd})
\begin{multline*}
\left[\mathcal D_{c,q,p}^{(k)}\frac{(az,a/z;q,p)_n}
{(cz,c/z;q,p)_n}\right]_{z=aq^{\frac k2}}\\=
(-1)^k(2a)^kq^{\binom k2/2}\frac{(q;q,p)_n(c/a,acq^{n-1};q,p)_k}
{(q;q,p)_{n-k}\,\ta(q;p)^k}
\left[\frac{(aq^{\frac k2}z,aq^{\frac k2}/z;q,p)_{n-k}}
{(cq^{\frac {3k}2}z,cq^{\frac {3k}2}/z;q,p)_{n-k}}\right]_{z=aq^{\frac k2}}\\=
(-1)^k(2a)^kq^{\binom k2/2}
\frac{(q,c/a,acq^{k-1};q,p)_k}{\ta(q;p)^k}\delta_{nk}.
\end{multline*}
The theorem now follows by applying $\mathcal D_{c,q,p}^{(j)}$ to both sides
of \eqref{exp} and then setting $z=aq^{\frac j2}$.
\end{proof}
\begin{example}
Let
\begin{equation*}
f(z)=\frac{(bz,b/z;q,p)_n}
{(cz,c/z;q,p)_n}.
\end{equation*}
Application of Theorem~\ref{thnew} in conjuction with \eqref{degl} gives
\begin{multline*}
f_k=\frac{(-1)^kq^{-k(k-1)/4}\,
\ta(q;p)^k}{(2a)^k(q,c/a,acq^{k-1};q,p)_k}\\\times
(-1)^k(2b)^kq^{\binom k2/2}
\frac{(q;q,p)_n(c/b,bcq^{n-1};q,p)_k}{(q;q,p)_{n-k}\,\ta(q;p)^k}
\frac{(abq^k,b/a;q,p)_{n-k}}{(acq^{2k},cq^k/a;q,p)_{n-k}}\\=
\frac{(ab,b/a;q,p)_n}{(ac,c/a;q,p)_n}\frac{\ta(acq^{2k-1};p)}{\ta(acq^{-1};p)}
\frac{(acq^{-1},c/b,bcq^{n-1},q^{-n};q,p)_k}{(q,ab,aq^{1-n}/b,acq^n;q,p)_k}q^k,
\end{multline*}
thus yielding Frenkel and Turaev's ${}_{10}V_9$ summation \eqref{10V9},
in the form
\begin{equation*}
\frac{(ac,c/a,bz,b/z;q,p)_n}
{(ab,b/a,cz,c/z;q,p)_n}={}_{10}V_9(acq^{-1};az,a/z,c/b,bcq^{n-1},q^{-n};q,p).
\end{equation*}
\end{example}
\begin{example}
Let
\begin{equation*}
R_n(z;b,c,d,e,f;q,p)\\
={}_{12}V_{11}(bcq^{-1};bz,b/z,d,e,f,bc^3q^{n-1}/def,q^{-n};q,p),
\end{equation*}
which is Spiridonov's~\cite{Sp2} elliptic extension of
Rahman's family of biorthogonal rational functions.
We have 
\begin{multline*}
\mathcal D_{c,q,p}^{(k)}R_n(z;b,c,d,e,f;q,p)\\
=\frac{(-1)^k(2b)^kq^{k(k+3)/4}\,(bc;q,p)_{2k}
(c/b,d,e,f,bc^3q^{n-1}/def,q^{-n};q,p)_k}
{\ta(q;p)^k\,(bc/d,bc/e,bc/f,bcq^n,defq^{1-n}/c^2;q,p)_k}\\\times
R_{n-k}(z;bq^{\frac k2},cq^{\frac{3k}2},dq^k,eq^k,fq^k;q,p).
\end{multline*}
Application of Theorem~\ref{thnew} now yields, after some computation,
the connection coefficient identity
\begin{multline}\label{concoeff}
R_n(z;b,c,d,e,f;q,p)\\=\sum_{k=0}^n\frac{(az,a/z;q,p)_k}
{(cz,c/z;q,p)_k}\frac{b^kq^k(bc;q,p)_{2k}
(c/b,d,e,f,bc^3q^{n-1}/def,q^{-n};q,p)_k}{a^k\,
(q,bc/d,bc/e,bc/f,bcq^n,defq^{1-n}/c^2;q,p)_k}\\\times
R_{n-k}(aq^{\frac k2};bq^{\frac k2},cq^{\frac{3k}2},dq^k,eq^k,fq^k;q,p).
\end{multline}
(Observe that the left-hand side of \eqref{concoeff} is independent of $a$.)

Now note that
\begin{equation}
R_m(c/f;b,c,d,e,f;q,p)=\frac{(bc,bc/de,c^2/df,c^2/ef;q,p)_m}
{(bc/d,bc/e,c^2/f,c^2/def;q,p)_m},
\end{equation}
due to Frenkel and Turaev's ${}_{10}V_9$ summation.
Letting $a\to c/f$ in \eqref{concoeff} gives, after some simplification,
\begin{multline}
R_n(z;b,c,d,e,f;q,p)\\=\frac{(bc,bc/de,c^2/df,c^2/ef;q,p)_n}
{(bc/d,bc/e,c^2/f,c^2/def;q,p)_n}\,R_n(z;c/f,c,d,e,c/b;q,p),
\end{multline}
 which is equivalent to 
Frenkel and Turaev's ${}_{12}V_{11}$ transformation in \eqref{12V11}.
\end{example}

\section{Outlook: Well-poised basic expansions}

As an outlook we sketch some details of our further investigations.
These concern infinite convergent expansions in the basic $p=0$ case.

For instance, using an extension of the well-poised basic
Taylor expansion theorem involving a remainder term, we
obtain, by using a symmetry argument, the following expansion:
\begin{multline*}
\frac{(cz/d,c/dz,cz/e,c/ez;q)_\infty}{(cz,c/z,c^2z/bde,c^2/bdez;q)_\infty}
=\frac{(cz/de,c/dez;q)_\infty}{(c^2z/bde,c^2/bdez;q)_\infty}
\sum_{k\ge 0}f_k\,\frac{(bz,b/z;q)_k}{(cz,c/z;q)_k}\\
+\frac{(bz,b/z;q)_\infty}{(cz,c/z;q)_\infty}\sum_{k\ge 0}g_k\,
\frac{(cz/de,c/dez;q)_k}{(c^2z/bde,c^2/bdez;q)_k}.
\end{multline*}
After the explicit computation of the coefficients $f_k$ and $g_k$
one recovers the nonterminating ${}_8\phi_7$ summation
(cf.\ \cite[Appendix~(II.25)]{GR}), in the form
\begin{multline*}
\frac{(cz/d,c/dz,cz/e,c/ez;q)_\infty}{(cz,c/z,c^2z/bde,c^2/bdez;q)_\infty}\\
=\frac{(cz/de,c/dez;q)_\infty}{(c^2z/bde,c^2/bdez;q)_\infty}
\sum_{k\ge 0}\frac{(1-bcq^{2k-1})}{(1-bcq^{-1})}
\frac{(bcq^{-1},d,e,c^2/deq,bz,b/z;q)_k}
{(q,bc/d,bc/e,bdeq/c,cz,c/z;q)_k}q^k\\
+\frac{(bz,b/z;q)_\infty}{(cz,c/z;q)_\infty}\sum_{k\ge 0}
\frac{(1-c^3q^{2k-1}/bd^2e^2)}{(1-c^3/bd^2e^2q)}\\\times
\frac{(c^3/bd^2e^2q,c/bd,c/be,c^2/deq,cz/de,c/dez;q)_k}
{(q,c^2/de^2,c^2/d^2e,cq/bde,c^2z/bde,c^2/bdez;q)_k}q^k.
\end{multline*}

To give another example, by expanding the ``quadratic'' infinite product
\begin{equation*}
f(z)=\frac{(azq,aq/z,b^2z/a,b^2/az;q^2)_\infty}
{(bz,b/z;q)_\infty}
\end{equation*}
in terms of the ``well-poised monomials''
\begin{equation*}
\frac{(az,a/z;q)_k}
{(bz,b/z;q)_k},
\end{equation*}
we recover a particular nonterminating ${}_8\phi_7$ summation,
namely Bailey's $q$-analogue of Watson's $_3F_2$ summation
(cf.\ \cite[Ex.~2.17(i)]{GR}), in the form
\begin{multline*}
\frac{(azq,aq/z,b^2z/a,b^2/az;q^2)_\infty}
{(bz,b/z;q)_\infty}
=\frac{(q,a^2q,b^2,b^2/a^2;q^2)_\infty}{(-ab,-b/a;q)_\infty}\\\times
\sum_{k\ge 0}\frac{(1+abq^{2k-1})}{(1+abq^{-1})}
\frac{(-abq^{-1},bq^{-\frac 12},-bq^{-\frac 12},-aq/b,az,a/z;q)_k}
{(q,-aq^{\frac 12},aq^{\frac 12},b^2q^{-1},bz,b/z;q)_k}
\left(\frac ba\right)^k.
\end{multline*}

A paper featuring these well-poised basic expansions is under preparation.

\bibliographystyle{amsalpha}

\end{document}